\documentclass[12pt]{amsart}
\usepackage{amsfonts,amsmath,amssymb,amscd,comment}

\newtheorem{theorem}{Theorem}[section]
\newtheorem{lemma}[theorem]{Lemma}
\newtheorem{corollary}[theorem]{Corollary}
\newtheorem{defn}{Definition}
\newtheorem{remark}[theorem]{Remark}

\newcommand{\lnr}{L_{n}^{\langle r \rangle}(x)}
\newcommand{\Lnr}{\mathcal{L}_{n}^{\langle r \rangle}(x)}
\newcommand{\Dt}{\Delta_{n}^{\langle r \rangle}}
\newcommand{\Q}{\mathbb{Q}}
\newcommand{\N}{\mathbb{N}}
\newcommand{\R}{\mathbb{R}}
\newcommand{\Z}{\mathbb{Z}}
\newcommand{\s}{\sigma}
\renewcommand{\a}{\alpha}

\begin{document}
\title[Irreducibility and Galois Groups of $L_{n}^{(-1-n-r)}(x)$]{Irreducibility and Galois Groups of Generalized Laguerre Polynomials $L_{n}^{(-1-n-r)}(x)$}
\author[Jindal, Laishram]{Ankita Jindal, Shanta Laishram}
\address{Department of Mathematics\\
IIT Delhi, New Delhi 110016, India}
\email{ankitajindal1203@gmail.com}
\author[Sarma]{Ritumoni Sarma}
\address{Stat-Math Unit, Indian Statistical Institute\\
7 S. J. S. Sansanwal Marg, New Delhi, 110016, India}
\email{shanta@isid.ac.in}
\address{Department of Mathematics\\
IIT Delhi, New Delhi 110016, India}
\email{ritumoni@maths.iitd.ac.in}

\dedicatory{Dedicated to Professor T. N. Shorey on his 70th birthday}
\begin{abstract}
We study the algebraic properties of Generalized Laguerre polynomials for negative integral values of a given parameter which is $L_{n}^{(-1-n-r)}(x)= \sum\limits_{j=0}^{n} \binom{n-j+r}{n-j} \frac{x^{j}}{j!}$
for integers $r\geq 0, n\geq 1$. For different values of parameter $r$, this family provides polynomials which are of great interest. Hajir conjectured that for integers $r\geq 0$ and $n\geq 1$, $L_{n}^{(-1-n-r)}(x)$ is an irreducible polynomial whose Galois group contains $A_n$, the alternating group on $n$ symbols. Extending earlier results of Schur, Hajir, Sell, Nair and Shorey, we confirm this conjecture for all $r\leq 60$.
We also prove that $L_{n}^{(-1-n-r)}(x)$ is an irreducible polynomial whose Galois group contains $A_n$ whenever $n>e^{r\left(1+\frac{1.2762}{{\rm log } r}\right)}$. 
\end{abstract}
\maketitle
\pagenumbering{arabic}
\pagestyle{headings}

\section{Introduction}
For an arbitrary real number $\a$ and a positive integer $n,$ the Generalized Laguerre Polynomials (GLP) is a family of polynomials defined by
	\begin{align*}
	L_{n}^{(\a)}(x) &= (-1)^{n} \sum\limits_{j=0}^{n} \binom{n+\a}{n-j} \frac{(-x)^{j}}{j!}.
	\end{align*}
The inclusion of the sign $(-1)^n$ is not standard. The corresponding monic polynomial is obtained as $\mathcal{L}_{n}^{(\a)}(x) = n! L_{n}^{(\a)}(x)$.
These classical orthogonal polynomials play an important role in various branches of analysis and mathematical physics and has been well studied. Schur \cite{schur1973a},
\cite{schur1973b} was the first to study the algebraic properties of these polynomials by proving that $L_{n}^{(\alpha)}(x)$ where $\alpha \in \{0,1,-n-1\}$ are irreducible.
For an account of results obtained on GLP, we refer to Hajir \cite{hajir2009} and Filaseta, Kidd and Trifonov \cite{fkt}.

In this paper, we study $\alpha$ at negative integral values via a parameter $r$. For integers $r\geq 0$, we consider
	\begin{align*}
	\lnr &:= L_{n}^{(-1-n-r)}(x)\\
	&= (-1)^{n} \sum\limits_{j=0}^{n} \binom{-1-r}{n-j} \frac{(-x)^{j}}{j!}\\
	&=  \sum\limits_{j=0}^{n} \binom{n-j+r}{n-j} \frac{x^{j}}{j!}.
	\end{align*}
	
By a factor of a polynomial, we always mean its factor over $\Q$. We observe that $\Lnr := n! \lnr= \sum\limits_{j=0}^{n} \binom{n}{j} (r+1) \dots (r+n-j) x^j$ is a monic polynomial with integer coefficients and $\lnr$ is irreducible if and only if $\Lnr$ is irreducible. Schur \cite{schur1973b} computed the discriminant of $\Lnr$ which is
	\begin{align*}
	\Dt=\prod\limits_{j=2}^{n}j^j(-1-n-r+j)^{j-1}.
	\end{align*}
Let $G_n(r)$ denote the Galois group of $\Lnr$ over $\Q$. Let $S_n$ denote the symmetric group on $n$ symbols and $A_n$, the alternating group on $n$ symbols. Schur \cite{schur1973a, schur1973b} and Coleman \cite{coleman1987} used two different techniques to prove that $L_{n}^{\langle 0 \rangle}(x)$ is irreducible and $G_n(0)=S_n$ for every $n$. Hajir \cite{hajir1995} proved that $L_{n}^{\langle 1 \rangle}(x)$ is irreducible and $G_n(1)$ is $A_n$ if $n \equiv 1 \pmod 4$ and is $S_n$, otherwise. Sell \cite{sell2004} proved that $L_{n}^{\langle 2 \rangle}(x)$ is irreducible and $G_n(2)$ is $A_n$ if $n+1$ is an odd square and is $S_n$, otherwise.

The irreducibility of $L_{n}^{\langle n \rangle}(x)$, also known as Bessel polynomials, was conjectured for all $n$ by Grosswald \cite{grosswald} and assuming his conjecture he proved that the Galois group is $S_n$ for every $n$. The irreducibility of all Bessel polynomials was proved, first for all but finitely many $n$ by Filaseta \cite{filaseta1995} and later for all $n$ by Filaseta and Trifonov \cite{filaseta_trifonov2002}.

Hajir \cite{hajir2009} conjectured that for integers $r\geq0$ and $n\geq1$, $\lnr$ is irreducible and $G_n(r)$ contains $A_n$. It was also proved in \cite{hajir2009} that if $r$ is a fixed integer in the range $0\leq r \leq 8$, then for all $n\geq 1$, $\lnr$ is irreducible and has Galois group containing $A_n$. This was extended by Nair and Shorey \cite{nair} who proved the following.\\

\noindent\textbf{Theorem A.}
	{\it For $n\geq 1$,
	\begin{itemize}
		\item [$(i)$] $\lnr$ is irreducible for $3 \leq r \leq 22$.
		\item [$(ii)$] For $9 \leq r \leq 22$, $G_n(r)=S_n$ unless $(n,r)\in \{(8,9), (12,13), (13,16), (16,17),$ $(17,18), (20,21)\}$ in which case $G_n(r)=A_n$. For $3 \leq r \leq 8$, $G_n(r)=S_n$ unless $(n,r)\in \{(2,3), (24,4),$ $ (4,5), (6,7), (7,8), (9,8), (2,8)\}$ or\\
		$r=3;$ $n \equiv 1 \pmod {24}$ and $\frac{n+2}{3}$ is a square\\
		$r=4;$ $n+2$ is a rational part of $(2+\sqrt{3})^{2k+1}$ where $k\geq 0$ is an integer\\
		$r=5;$ $n+3$ is a rational part of $(4+\sqrt{15})^{2k+1}$ where $k\geq 0$ is an integer\\
		in which case $G_n(r)=A_n$.
	\end{itemize}
}

We further extend this work to confirm the conjecture of Hajir for all $r\leq 60$. We prove
\begin{theorem} \label{mainthm1}
	For $n \geq 1$ and $23 \leq r \leq 60$, we have
	\begin{itemize}
		\item [$(i)$] $\lnr$ is irreducible.
		\item [$(ii)$] $G_n(r)=S_n$ unless $(n,r)\in \{(4,24),(5,28),(24,25),(25,24),(28,23),(28,29)$, $(32,33), (33,36), (36,37),(40,41),(44,45),(48,49),(48,51),(49,48),(49,50)$, \\ $(52,53)$, $(56,57)\}$ in which case $G_n(r)=A_n$.
	\end{itemize}
\end{theorem}
The proof of Theorem \ref{mainthm1} is given in Sections \ref{proofpart1} and \ref{proofpart2}. We see that Theorem \ref{mainthm1} considerably extends earlier results of \cite{hajir2009} and \cite{nair}.
The new ingredients in the proof are Lemma 3.1 which arise from clever and important observations on prime divisors of $n$ and $\binom{n+r}{r}$ and Lemmas 3.4-3.6 which arise from an application of p-adic Newton polygons. These results are general in nature and make our computations much less. In fact, for checking irreducibility of $\lnr$, we need to exclude factors of degrees up to $3$ which can be handled easily. The observations also imply the following result which improves the bound for $n$ given by Hajir \cite{hajir2009} and Nair and Shorey \cite{nair}.

\begin{theorem} \label{mainthm2}
	For $n \geq 1$ and $r\geq 0$, $\lnr$ is irreducible and $G_n(r)$ contains $A_n$ if
	$$n>e^{r\left(1+\frac{1.2762}{{\rm log } r}\right)}.$$
\end{theorem}
We prove Theorem \ref{mainthm2} in Section \ref{proof of mainthm2}.

The computations in this paper are carried out with SAGE except for computing a few Galois groups in Section 5 for which MAGMA online is used.

\section{Preliminaries}
Henceforth, we always use $p$ for a prime and $n$, $r$ for integers with $r\geq 0$ and $n \geq 1$ unless otherwise specified.
\begin{defn}
	The {\it p-adic valuation} of an integer $m$ with respect to $p$, denoted by $\nu_p(m)$, is defined as
		\begin{align*}
		\nu_p(m)&=\begin{cases}
				\max\{k:p^k|m\}& {\rm if} \ m\ne 0, \\
				\infty& {\rm if} \ m=0.
				\end{cases}
		\end{align*}
\end{defn}
\begin{defn}
	Let $m$ be a positive integer. Let $m=m_0+m_1p+\cdots+m_tp^t$ with $m_t\ne 0$  be the p-adic representation of $m$. We define $\s_p(m):=m_0+m_1+\cdots+m_t$.
\end{defn}
For integers $m\geq 1$ and $t\geq 0$, we have
	\begin{align*}
	\nu_p(m!)&=\frac{m-\s_p(m)}{p-1},\\
	{\rm and~} \nu_{p} \left( \binom{m}{t} \right) &= \frac {\s_p(t) + \s_p(m-t) -\s_p(m) }{p-1}.
	\end{align*}
These are well known results of Legendre \cite{legendre}.

\begin{defn}
	Let $f(x)= \sum\limits_{j=0}^{n}a_{j}x^{j} \in \Z[x]$ with $a_{o}a_{n} \neq 0$. We consider the set
	$$S=\{(0,\nu_p(a_n)), (1,\nu_p(a_{n-1})), \dots, (n,\nu_p(a_0)) \}$$
	consisting of points in the extended plane $\R^{2} \cup \{ \infty \}$ . The polygonal path formed by the lower edges along the convex hull of $S$ is called the {\it Newton polygon} associated to $f(x)$ with respect to prime $p$ and is denoted by $NP_p(f(x))$.
\end{defn}
It can be observed that the left-most edge has one end point being $(0,\nu_p(a_n))$ and the right-most edge has $(n,\nu_p(a_0))$ as an end point. The end points of every edge belong to the set $S$. Thus every point in $S$ lies either on or above the line obtained by extending such an edge. In particular, if $(i,\nu_p(a_{n-i}))$ and $(j,\nu_p(a_{n-j}))$ are the two end-points of such an edge, then every point $(u,\nu_p(a_{n-u}))$ with $i < u < j$ lies on or above the line passing through $(i,\nu_p(a_{n-i}))$ and $(j,\nu_p(a_{n-j}))$. Also the slopes of the edges are always increasing when calculated from the left-most edge to the right-most edge.

The following result is due to Dumas \cite{dumas}.
\begin{lemma} \label{dumas}
	Let $g(x), h(x) \in \Z[x]$ with $g(0)h(0) \neq 0$, and let $p$ be a prime. Let $k$ be a non-negative integer s.t. $p^{t}$ divides the leading coefficient of $g(x)h(x)$ but $p^{t+1}$ does not. Then the edges of the Newton Polygon for $g(x)h(x)$ with respect to $p$ can be formed by constructing a polygonal path beginning at $(0,t)$ and using translates of the edges in the Newton Polygons for $g(x)$ and $h(x)$ with respect to $p$ (using exactly one translate for each edge). Necessarily, the translated edges are translated in such a way as to form a polygonal path with the slopes of the edges increasing. 
\end{lemma}

We also need the following result due to Filaseta \cite[Lemma 2]{filaseta1995} which is a consequence of Lemma \ref{dumas}.

\begin{lemma} \label{application of np}
	Let $k$ and $l$ be integers with $k>l \geq 0$. Suppose $g(x)= \sum\limits_{j=0}^{n} b_{j} x^{j} \in \Z[x]$ and $p$ is a prime such that
	$p \nmid b_{n}$, $p|b_{j}$ for all $j \in \{0,1, \dots, n-l-1\}$ and the right-most edge of the Newton polygon for $g(x)$ with respect to $p$ has slope $< \frac{1}{k}$. Then for any integers $a_{0}, a_{1}, \dots , a_{n}$ with $|a_{0}|=|a_{n}|=1$, the polynomial $f(x)= \sum\limits_{j=0}^{n} a_{j} b_{j} x^{j}$ cannot have a factor with degree in the interval $[l+1, k]$.
\end{lemma}

In this paper, we use Lemma \ref{application of np} with $a_0=a_1=\cdots =a_n=1$ always.
\begin{defn}
	Given $f\in \Q[x]$, we define the Newton Index of $f$, denoted by $\mathcal{N}_f$, to be the least common multiple of the denominators (in lowest terms) of all slopes of $NP_p(f(x))$ as $p$ ranges over all primes.
\end{defn}

The following results by Hajir \cite[Theorem 2.2]{hajir2005} are used for calculating the Galois groups of polynomials.
\begin{lemma}\label{newtonindex}
	Given an irreducible polynomial $f\in \Q[x]$, $\mathcal{N}_f$ divides the order of the Galois group of $f$. Moreover, if $\mathcal{N}_f$ has a prime divisor $q$ in the range $\frac{n}{2} < q < n-2$, where $n$ is the degree of $f$, then the Galois group of $f$ contains $A_n$.
\end{lemma}

As a consequence of Lemma \ref{newtonindex}, Hajir \cite[Theorem 5.4]{hajir2009} proved the following result.
\begin{lemma} \label{lemma for gp}
  Let $\lnr$ be irreducible.
	\begin{itemize}
		\item[$(i)$] If there exists a prime $p$ satisfying $\frac{n+r}{2}<p<n-2$, then $G_n(r)$ contains $A_n$.
		\item[$(ii)$] If $n \geq \max\{48-r,8+\frac{5r}{3}\}$, then $G_n(r)$ contains $A_n$.
		\item[$(iii)$] If $G_n(r)$ contains $A_n$, then
		$$G_n(r)=
		\begin{cases}
		A_n & {\rm if} \ \Dt \ {\rm is \ a \ square},\\
		S_n & {\rm otherwise.}
		\end{cases}
		$$
	\end{itemize}
\end{lemma}
If $\Lnr$ is reducible, it has at least one factor with degree $\in [1,\frac{n}{2}]$. Thus from now onwards, whenever we consider a factor of degree $k$ of $\Lnr$, we mean a factor of degree $k$ with $1 \leq k \leq \frac{n}{2}.$

For fixed integers $r \geq 0$ and $n \geq 1$, we write $n=n_{0}n_{1}$ where
	\begin{align*}
	n_{0} := \prod\limits_{p|n,\ p\nmid \binom{n+r}{r}} p^{\nu_{p}(n)} \textrm{ and } n_{1} := \prod\limits_{p|\gcd(n,\binom{n+r}{r}) } p^{\nu_{p}(n)}.
	\end{align*}

The following result is contained in the first line of the proof of Hajir \cite[Lemma 4.1]{hajir2009}
\begin{lemma} \label{degree}
	Every factor of $\lnr$ has degree divisible by $n_{0}.$
\end{lemma}

The next three results are due to Nair and Shorey \cite[Corollary 3.2, Corollary 3.3 and Lemma 2.10]{nair}.
\begin{lemma} \label{rk1}
	Assume that $\lnr$ has a factor of degree $k \geq 2.$ Then $r>1.63k$.
\end{lemma}

\begin{lemma} \label{rk2}
	Assume that $\lnr$ has a factor of degree $k \geq 2.$ Then
	$$r>\min \{104,3.42k+1\}.$$
\end{lemma}

\begin{lemma} \label{small values}
	For $n\leq127$ and $r\leq103$, $\lnr$ is irreducible.
\end{lemma}

We also need the following statement used in \cite{nair} and we give a proof here.
\begin{lemma} \label{ppower}
	For $p|n_{1}$, we have $p^{\nu_{p}(n)} \leq r.$	
\end{lemma}
\begin{proof}
	Write $n=p^{e}d,$ where $d$ is coprime to $p$ such that $p^{e}>r.$ We will show that $\nu _{p} \left( \binom{n+r}{r} \right) = 0.$
	
	Let $r=r_{e-1}p^{e-1}+\cdots + r_{1}p + r_{0}$ be the $p$-adic representation of $r$. Then $n+r = d p^{e} + r_{e-1} p^{e-1} + \cdots + r_{1} p + r_{0}.$ So we have $\s_p(n) = \s_p(d),~ \s_p(r) = r_{e-1} + \cdots + r_{1} + r_{0}$ and $\s_p(n+r) = \s_p(d) + r_{e-1} + \cdots + r_{1} + r_{0}.$ Thus $\nu_{p} \left( \binom{n+r}{r} \right) =  \frac {\s_p(n) + \s_p(r) -\s_p(n+r) }{p-1} = 0.$
\end{proof}

The following result is due to Harborth and Kemnitz \cite{harborthkemnitz}.
\begin{lemma} \label{harborthkemnitz}
	There exists a prime $p$ satisfying :
	\begin{itemize}
		\item [$(a)$] $x<p<\frac{6}{5}x$ for real $x\geq 25$,
		\item [$(b)$] $x<p\leq \frac{11}{10}x$ for real $x\geq 116$.
	\end{itemize}
\end{lemma}

For real numbers $x>1$, we denote
	\begin{align*}
	\pi(x)&=\sum\limits_{p\leq x} 1.
	\end{align*}

We need the following result due to Dusart \cite{dusart} for the proof of Theorem \ref{mainthm2}.
\begin{lemma} \label{pi}
	We have
	\begin{align*}
	\pi(x)&\leq \frac{x}{\log x} \left(1+\frac{1.2762}{\log x}\right) \ \ {\it for ~real} \ x>1.
	\end{align*}
\end{lemma}

\section{Lemmas for the proof of Theorem \ref{mainthm1}}
For the proof of Theorem \ref{mainthm1}, we use a number of results which we record here as lemmas and corollaries. These results are general in nature and valid for any positive integers $n$ and $r$.

\begin{lemma}\label{3.1}
	Let $p|n_{1}$ and $r < p^2$. Then $\left\lfloor\frac{r}{p}\right\rfloor \geq 1$ and
	\begin{align*}
	\frac{n}{p} &\equiv -j \pmod p \ {\it for \ some} \ j \ {\it with} \ 1 \leq j \leq \left\lfloor\frac{r}{p}\right\rfloor.
	\end{align*}
\end{lemma}
\begin{proof}
	Since $p|n_{1}$ and $r < p^2$, it follows from Lemma \ref{ppower} that $\left\lfloor\frac{r}{p}\right\rfloor \geq 1$ and $\nu_{p}(n_{1}) = 1$. We can write $n=pd,$ where $d$ is coprime to $p$ and $r = r_1 p + r_0,$ where $1 \leq r_1<p,~ 0 \leq r_0 < p.$ Then $n+r = p(d+r_1) + r_0.$ So we have $\s_p(n) = \s_p(d),~ \s_p(r) = r_1 + r_{0}$ and $\s_p(n+r) = \s_p(d+r_1) + r_{0}.$ Therefore
	\begin{align*}
		1\leq \nu_{p} \left( \binom{n+r}{r} \right) &= \frac {\s_p(n) + \s_p(r) -\s_p(n+r) }{p-1}\\
		&= \frac {\s_p(d) + r_1 -\s_p(d+r_1) }{p-1}\\
		&= \nu_{p} \left( \binom{d+r_1}{r_1} \right)\\
		&= \nu_{p} \left( \frac{(d+1)(d+2)\cdots(d+r_1)}{r_1!} \right)\\
		&= \nu_{p} ((d+1)(d+2)\cdots(d+r_1))~ (\textrm{since } r_1<p) \\
		&= \nu_{p} (d+j)~ \textrm{for exactly one } j \textrm{ with } 1 \leq j \leq r_1.
	\end{align*}
	Since $r_1=\left\lfloor\frac{r}{p}\right\rfloor<p$, we have $\frac{n}{p} \equiv -j \pmod p,$ for some $1 \leq j \leq \left\lfloor\frac{r}{p}\right\rfloor.$
\end{proof}

\begin{corollary} \label{cong}
	If $p|n_{1}$ and $r<p^2,$ then $d+\left\lfloor\frac{r}{p}\right\rfloor\geq p$ where $d\equiv \frac{n}{p}  \pmod p$ with $1\leq d<p$.
\end{corollary}

For the remaining part of this paper, we need the following notation and remark.

\begin{remark} \label{remark}
	\par For $0 \leq j \leq n$, we define $b_j := \binom{n}{j} (r+1) \cdots (r+j)$. The Newton polygon for $\Lnr=\sum\limits_{j=0}^{n} b_{n-j}x^j$ with respect to $p$ is given by the lower edges along the convex hull of the points $(j,\nu_p(b_j))$ for $0 \leq j \leq n.$ Thus the slope of the right-most edge of $NP_p(\Lnr)$ is $M_p=\max\limits_{1 \leq j \leq n} \{ \mu_j \} $ where
	\begin{align*}
	\mu_j &:=\frac{\nu_p(b_n)-\nu_p(b_{n-j})}{j} \\
	&=  \frac{\nu_p((n+r)!) - \nu_p((n+r-j)!) - \nu_p(\binom{n}{j})}{j}\\
	&= \frac{j-\s_p(n+r)+\s_p(n+r-j)}{(p-1)j} - \frac{\s_p(j)+\s_p(n-j)-\s_p(n)}{(p-1)j}\\
	&= \frac{j- \s_p(j)}{(p-1)j}+\frac{\s_p(n)+\s_p(r)-\s_p(n+r)}{(p-1)j} -\frac{\s_p(n-j)+\s_p(r)-\s_p(n+r-j)}{(p-1)j}\\
	&= \frac{j-\s_p(j)}{(p-1)j} + \frac{1}{j} \nu_p \left( \binom{n+r}{r} \right)-\frac{1}{j} \nu_p \left( \binom{n+r-j}{r} \right)\\
	&\leq \frac{j-\s_p(j)}{(p-1)j} + \frac{1}{j} \nu_p \left( \binom{n+r}{r} \right) ({\rm since~ } \nu_p \left( \binom{n+r-j}{r} \right)\geq0 ).
	\end{align*}
\end{remark}

\begin{lemma} \label{step1}
	Let $p = p_{\pi(n)} = n - k_n$ be the largest prime less than or equal to $n$ with $r + k_n < p$. Then $\Lnr$ cannot have a factor with degree  $>k_n$.
\end{lemma}
\begin{proof}
	Clearly $p \nmid b_0$. Since $p\mid n(n-1)\cdots (n-k_n)$, $p | \binom{n}{j}$ for $k_n+1 \leq j < p$. Also, $p \mid (r+1) \cdots (r+j)$ for $j \geq p$. Thus $p|b_j$ for $k_n +1 \leq j \leq n$.
	
	Note that  $r + k_n < p$ implies $p \nmid (r+1) \cdots (r+k_n)$ and $p\nmid n(n-1)\cdots (n-k_n+1)$. Thus $p \nmid (r+1) \cdots (r+j)$ and $p \nmid \binom{n}{j}$ for $1\leq j \leq k_n$. Therefore $p \nmid b_j$ for $1\leq j \leq k_n$.
	
	Next $r+n=r+k_n+p<2p$ implies $\nu_p(b_n) = \nu_p ((r+1) \cdots (r+n))=1$. Hence the vertices of the first edge of the Newton polygon are $(0,0)$ and $(k_n,0)$ and the slope of the right-most edge is		
		\begin{align*}
		\max\limits_{k_n \leq j < n} \left\{ \frac{\nu_p(b_n)-\nu_p(b_{j})}{n-j} \right\}.
		\end{align*}
 For $k_n<j<n$, we have $p|b_j$ implying $\nu_p(b_{j})\geq 1$. Hence $\nu_p(b_n)-\nu_p(b_{j})\leq 1-1=0$ for $k_n<j<n$. For $j=k_n$, we have 	
		\begin{align*}
		\frac{\nu_p(b_n)-\nu_p(b_{k_n})}{n-k_n}=\frac{1}{n-k_n} = \frac{1}{p}.
		\end{align*}
	Thus we have
		\begin{align*}
		\max\limits_{k_n \leq j < n} \left\{ \frac{\nu_p(b_n)-\nu_p(b_{j})}{n-j} \right\} \leq \frac{1}{p}<\frac{2}{n}
		\end{align*}
  since $p>\frac{n}{2}$. Therefore, by Lemma \ref{application of np}, $\Lnr$ cannot have a factor with degree in the interval $[k_n+1, \frac{n}{2}]$ and the assertion follows.
\end{proof}

\begin{lemma}\label{step2}
	Let $l_n \in[1,k_n]$ be the least positive integer such that there exists $p$ with $p | (n-l_n)$, $p > k_n$ and $\nu_p \left( \binom{n+r}{r} \right)=0$.
	Then $\Lnr$ cannot have a factor with degree in the interval $[l_n+1, k_n].$
\end{lemma}
\begin{proof}
	Clearly $p \nmid b_0$. Since $p\mid n(n-1)\cdots (n-l_n)$, $p | \binom{n}{j}$ for $l_n+1 \leq j < p$. Also $p \mid (r+1) \cdots (r+j)$ for $j \geq p$. Thus $p|b_j$ for $l_n +1 \leq j \leq n$.
	
	From Remark \ref{remark}, the slope of the right-most edge of $NP_p(\lnr)$ is equal to $M_P \leq \max\limits_{1 \leq j \leq n} \left\{ \frac{j-\s_p(j)}{(p-1)j} + \frac{1}{j} \nu_p \left( \binom{n+r}{r} \right) \right\}$.
	
	Note that $\frac{j-\s_p(j)}{(p-1)j} \leq 0$ if $j \leq p-1$ and $\frac{j-\s_p(j)}{(p-1)j} <\frac{1}{p-1}$ if $j \geq p$.
	Since $p > k_n$ and $\nu_p \left( \binom{n+r}{r} \right)=0$, we have
		\begin{align*}	
		M_p < \frac{1}{k_n}.
		\end{align*}
	Therefore, by Lemma \ref{application of np}, $\Lnr$ cannot have a factor with degree in the interval $[l_n+1, k_n].$
\end{proof}

\begin{lemma}\label{step3}
	Let $i$ be a positive integer such that $p|n(n-1)\cdots(n-i+1)(r+1)\cdots(r+i)$ and let $ \nu_p \left( \binom{n+r}{r} \right) = u$. Then $\Lnr$ cannot have a factor of degree equal to $i$ if any one of the following conditions holds:
	\begin{itemize}
		\item [$(a)$] $u=0$ and $p > i$,
		\item [$(b)$] $u>0,~ p>2$ and $\max \{\frac{u+1}{p}, \frac{\nu_p(n+r-z_0)-\nu_p(n)}{z_0+1} \} <\frac{1}{i}$, where $z_0\equiv n+r \pmod p$ with $0 \leq z_0 <p$.
	\end{itemize}	
\end{lemma}
\begin{proof}
	Clearly $p \nmid b_0$. If $p|(r+1)\cdots(r+i),$ then $p|b_j$ for $i \leq j \leq n$. If $p\nmid (r+1)\cdots(r+i),$ then $p|n(n-1)\cdots(n-i+1)$
	implies $p | \binom{n}{j}$ for $i \leq j < p$. Also $p|(r+1) \cdots (r+j)$ for $j \geq p$. Thus $p|b_j$ for $i \leq j \leq n$.
	
	From Remark \ref{remark}, the slope of the right-most edge of $NP_p(\lnr)$ is $M_p=\max\limits_{1 \leq j \leq n} \{ \mu_j \} $ where
		\begin{align*}
		\mu_j \leq \frac{j-\s_p(j)}{(p-1)j} + \frac{u}{j}.
		\end{align*}
		
\noindent $(a) \ u=0$ and $p>i$. For $1\leq j\leq n$, we have
		\begin{align*}
		\mu_j \leq  \frac{j-\s_p(j)}{(p-1)j} < \frac{1}{p-1} \leq \frac{1}{i}.
		\end{align*}
	
\noindent $(b) \ u>0$ and $p>2$. We have
		\begin{align*}
		\mu_j &= \frac{\nu_p((n+r)!) - \nu_p((n+r-j)!) - \nu_p(\binom{n}{j})}{j} \\
		&= \frac{\nu_p((n+r)\cdots (n+r-j+1)) - \nu_p(\binom{n}{j})}{j}.
		\end{align*}
	For $1\leq j < p$, we have
		\begin{align*}
		\mu_j &\leq
			\begin{cases}
			0 & \textrm{if } j \leq z_0\\
			\frac{\nu_p(n+r-z_0)-\nu_p(n)}{j} & \textrm{if } j >z_0\\
			\end{cases}\\
		&\leq \frac{\nu_p(n+r-z_0)-\nu_p(n)}{z_0+1}.
		\end{align*}
	For $p \leq j < p^2$, we have
		\begin{align*}
		\mu_j \leq \frac{j-\s_p(j)}{(p-1)j} + \frac{u}{j} \leq \frac{1}{p}+ \frac{u}{p} = \frac{u+1}{p}.
		\end{align*}
	For $j \geq p^2,$ since $p>2$, we have
		\begin{align*}
		\mu_j \leq \frac{j-\s_p(j)}{(p-1)j} + \frac{u}{j} < \frac{1}{p-1}+ \frac{u}{p^2}<\frac{u+1}{p}.
		\end{align*}
	Thus, by the assumption on (b), for $1 \leq j\leq n$,
		\begin{align*}
		\mu_j \leq \max \left\{\frac{u+1}{p}, \frac{\nu_p(n+r-z_0)-\nu_p(n)}{z_0+1} \right\} <\frac{1}{i}.
		\end{align*}
	Hence $M_p < \frac{1}{i}$ and therefore, by Lemma \ref{application of np}, $\Lnr$ cannot have a factor of degree $i$.
\end{proof}

We need the following three lemmas for describing the Galois groups of $\lnr$. The third lemma is computational.

\begin{lemma} \label{our result on gp}
	Given that $\Lnr$ is irreducible, if there is a prime $p$ with $\frac{n}{2}<p<n-2$ and $r<p$, then $G_n(r)$ contains $A_n$.
\end{lemma}
\begin{proof} Let $n_0=n-p$ and $r_0=p-r$. For $1 \leq j \leq n$, we have
	\begin{align*}
	\nu_p\left(\binom{n}{j}\right)=\nu_p\left(\frac{n(n-1)\cdots(n-j+1)}{j!}\right)=
		\begin{cases}
		1 & \textrm{if } n_0<j<p,\\
		0 & \textrm{otherwise.}
		\end{cases}
	\end{align*}
	
	First assume that $r+n<2p$. Note that $r_0>n_0$ and $r_0+p=r_0+n-n_0>n$. Thus $r+r_0=p$ is the only multiple of $p$ in the product $(r+1)(r+2)\cdots(r+n)$. So for $1 \leq j \leq n$, we have
		\begin{align*}
		\nu_p((r+1)(r+2)\cdots(r+j))=
			\begin{cases}
			0 & \textrm{if } j<r_0,\\
			1 & \textrm{otherwise}.
			\end{cases}
		\end{align*}
	Therefore $NP_p(\Lnr)$ is given by the lower edges along the convex hull of the points:
	$$(0,0),\dots,(n_0,0),(n_0+1,1),\dots,(r_0-1,1),(r_0,2),\dots,(p-1,2),(p,1),\dots,(n,1).$$
	Thus the vertices of $NP_p(\Lnr)$ are $(0,0),(n_0,0)$ and $(n,1)$. Hence $\frac{1}{p}$ is a slope of $NP_p(\Lnr)$ and it follows from Lemma \ref{newtonindex} that $G_n(r)$ contains $A_n$.
	
	Next assume that $r+n \geq 2p$. Since $r+n<3p$, $r+r_0=p$ and $r+r_0+p=2p$ are the only multiples of $p$ in the product $(r+1)(r+2)\cdots(r+n)$. So for $1 \leq j \leq n$, we have
		\begin{align*}
		\nu_p((r+1)(r+2)\cdots(r+j))=
			\begin{cases}
			0 & \textrm{if } j<r_0,\\
			1 & \textrm{if } r_0\leq j<r_0+p,\\
			2 & \textrm{if } j\geq r_0+p.
			\end{cases}
		\end{align*}
	Therefore in this case $NP_p(\Lnr)$ is given by the lower edges along the convex hull of the points:
	$$(0,0),\dots,(r_0-1,0),(r_0,1),\dots,(r_0+p-1,1),(r_0+p,2),\dots,(n_0,2),(n_0+1,3),\dots,$$
	$$(p-1,3),(p,2),\dots,(n,2).$$
	Thus the vertices of $NP_p(\Lnr)$ are $(0,0),(r_0-1,0),(r_0+p-1,1)$ and $(n,2)$. Hence $\frac{1}{p}$ is one of the slopes of $NP_p(\Lnr)$ and it follows from Lemma \ref{newtonindex} that $G_n(r)$ contains $A_n$.
\end{proof}

\begin{lemma} \label{squares}
	Let $m \geq 197$ be an odd integer and let $t\leq 60$ be an even integer. Then the product of any two distinct terms in the set $\{m+2, m+4,\dots, m+t\}$ cannot be a square.
\end{lemma}
\begin{proof}
	Suppose $(m+2i)(m+2j)$ is a square with $1 \leq i<j \leq \frac{t}{2}$. We may assume $m+2i=ax^2$ and $m+2j=ay^2$ where $y-x\geq 2$.
	Then $t-2 \geq 2(j-i)=a(y-x)(y+x)\geq 2a(y+x) \geq 4ax$. Therefore $x \leq ax \leq \lfloor\frac{t-2}{4}\rfloor\leq \lfloor\frac{60-2}{4}\rfloor=14$. Hence $m \leq ax^2\leq (14)^2$ which implies $m\leq 195$, a contradiction.	
\end{proof}

\begin{lemma}\label{primefact}
	There is a prime in every set of 20 consecutive positive integers each $\leq 1148$.
\end{lemma}

\section{Irreducibility of $\lnr$: Proof of Theorem \ref{mainthm1}$(i)$} \label{proofpart1}
In this section, we give proof of Theorem \ref{mainthm1}$(i)$ by showing that $\lnr$ is irreducible for each $23\leq r\leq 60$ and $n \geq 1$.
Recall that for fixed integers $r \geq 0$ and $n \geq 1$, $n=n_{0}n_{1}$ where
$$n_{0} := \prod\limits_{p|n,\ p\nmid \binom{n+r}{r}} p^{\nu_{p}(n)} \textrm{ and } n_{1} := \prod\limits_{p|\gcd(n,\binom{n+r}{r}) } p^{\nu_{p}(n)}.$$

Let $23\leq r \leq 60$ and $n\geq 1$ be integers. Suppose $\lnr$ has a factor of degree $k$. By Lemma \ref{degree}, we have $n_0|k$. So if $n_0\geq 2$, then $k\geq 2$ and thus Lemma \ref{rk2} implies $r>3.42k+1$, i.e., $n_0\leq k < \frac{r-1}{3.42}$. Therefore we have $1 \leq n_0 \leq \left\lfloor\frac{r-1}{3.42}\right\rfloor$ for each value of $r$.

Fix $r$ with $23\leq r \leq 60$. For each $n_0$, we have
	\begin{align*}
	\{n=n_0n_1: p^{\nu_p(n_1)} \leq r~ \forall p \} \subseteq \{n: p^{\nu_p(n)} \leq r~ \forall p \}.
	\end{align*}
Since $\left\lfloor\frac{r-1}{3.42}\right\rfloor\geq {\rm max}\{n_0,\sqrt{r}\}$, if $p|n$ with $p>\left\lfloor\frac{r-1}{3.42}\right\rfloor$, then $p|n_1$ and $r<p^2$.
Thus, by Lemma \ref{small values}, Lemma \ref{ppower} and Corollary \ref{cong}, it is enough to check irreducibility of $\lnr$ for $n \in H_r$ where
	\begin{small}
	\begin{align*}
	H_r=\{n\in \N: n > 127 \textrm{ and for each } p|n,~  p^{\nu_p(n)} \leq r \textrm{ and if } p > \left\lfloor\frac{r-1}{3.42}\right\rfloor \textrm{then } d+\left\lfloor\frac{r}{p}\right\rfloor\geq p \}
	\end{align*}
	\end{small}
where $1\leq d<p$ and $d \equiv \frac{n}{p} \pmod {p}$. (Note that $d \ne 0$ since if $p > \left\lfloor\frac{r-1}{3.42}\right\rfloor$, then $p^2\nmid n$).

For each $n \in H_r$, we compute $k_n$ and $l_n$ (defined respectively in Lemma \ref{step1} and Lemma \ref{step2}). We find that $l_n \leq 3$ for each $n \in H_r$ and it follows that $k\leq l_n \leq 3$. For $1\leq i \leq 3$, we define $H_{i,r}=\{n\in H_r: l_n \geq i\}$. To obtain a contradiction, we need to prove non-existence of a factor of degree $i$ for each $n \in H_{i,r}$. For this we use Lemma \ref{step3} and we are left with $(n,r)\in T$ for which $\lnr$ may have a factor of degree 1, where $T$ is given by
	\begin{align*}
	T=\{&(144,23),(144,25),(144,26),(144,51),(144,53),(216,29),(216,31), (216,42),\\
	&(216,44),(216,47),(216,49),(216,53),(216,59),(240,35), (288,40),(288,41),\\
	&(288,47),(288,48),(288,51),(288,53),(312,26),(600,26),(720,31),(1440,35),\\
	&(4320,55)\}.
	\end{align*}

Observe that if $p|n$, then $p | \binom{n}{j}$ for $1 \leq j < p$. Also $p|(r+1) \cdots (r+j)$ for $j \geq p$. Thus if $p|n$, then $p|b_j$ for all $1 \leq j \leq n$. Since $2|n$ and $3|n$ for each $n$ given in $T$, to remove the existence of a factor of degree $1$, by Lemma \ref{application of np}, it suffices to show that the slope of the right-most edge of $NP_p(\lnr)$, for either $p=2$ or $p=3$, is less than $1$.

By Remark \ref{remark}, it suffices to show that $\mu_j<1$ for each $1\leq j\leq n$, for either $p=2$ or $p=3$, where
	\begin{align}
	\mu_j &= \frac{\nu_p((r+n)(r+n-1)\cdots(r+n-j+1)) - \nu_p(\binom{n}{j})}{j}. \label{eq} 
	\end{align}
By Remark \ref{remark} again, we have
	\begin{align*}
	\mu_j &\leq \frac{j-\s_p(j)}{(p-1)j} + \frac{1}{j} \nu_p \left( \binom{n+r}{r} \right).
	\end{align*}
It can be easily observed that
	\begin{align*}
	\frac{j-\s_p(j)}{(p-1)j} + \frac{1}{j} \nu_p \left( \binom{n+r}{r} \right)<1
	\end{align*}
if and only if
	\begin{align}
	(p-1)\nu_p \left( \binom{n+r}{r} \right) &< (p-2)j+\s_p(j). \label{ineq1}
	\end{align}

Let $(n,r)\in T\setminus\{(216,29),(4320,55)\}$. We take $p=3$. In this case, the inequality \eqref{ineq1} is equivalent to
\begin{align}
2\nu_3 \left( \binom{n+r}{r} \right) &< j+\s_3(j). \label{ineq2}
\end{align}
For each $(n,r)\in T\setminus\{(216,29),(4320,55)\}$, we have $\nu_3 \left( \binom{n+r}{r} \right) \leq4$. Thus \eqref{ineq2} holds for $j\geq8$. For $j<8$, we verify that $\mu_j<1$ by exact computation of $\mu_j$ using \eqref{eq}.\\
 
Let $(n,r)\in \{(216,29),(4320,55)\}$. Suppose $x+a$ is a factor of $\Lnr$. Observe that $\Lnr$ is a monic polynomial whose coefficients are positive integers and hence the root $-a$ is a negative integer, i.e., $a \in \Z^+$. Note that for any prime $p$, $NP_p(x+a)$ consists of exactly one edge joining $(0,0)$ and $(1,\nu_p(a))$ which has slope $\nu_p(a)$ and therefore it follows from Lemma \ref{dumas} that $\nu_p(a)$ is the slope of an edge in $NP_p(\Lnr)$. Thus the non-negative integral slopes of $NP_p(\Lnr)$ are the only possible choices of $\nu_p(a)$. Consider the set 
$$I_p=\{\textrm{non-negative integral slopes of } NP_p(\Lnr) \}.$$
Note that for any prime $p$ such that $I_p \subseteq \{0\}$, we have $p\nmid a$. Therefore we may restrict to $p$ such that $I_p\cap \Z^+ \ne \phi$. 

Next we claim that for any prime $p$, we have $0\in I_p$ if and only if $p\nmid n(r+1)$. In fact, if there is an edge of slope $0$ in $NP_p(\Lnr)$, then we must have $\nu_p(n(r+1))=\nu_p(b_1)=0$. On the other hand, if $p|(r+1),$ then $p|b_j$ for all $1 \leq j \leq n$. If $p\nmid (r+1),$ then $p|n$ implies $p | \binom{n}{j}$ for $1 \leq j < p$. Also $p|(r+1) \cdots (r+j)$ for $j \geq p$. Thus $p|b_j$ for all $1 \leq j \leq n$. This implies $\nu_p(b_j)>0$ for all $1 \leq j \leq n$. Since $b_0=1$, the first point of $NP_p(\Lnr)$ is $(0,0)$ and hence it follows that there is no edge of slope $0$ in $NP_p(\Lnr)$. This proves the claim. We will use this claim without mentioning.

Now we determine the positive integral slopes of $NP_p(\Lnr)$ in the following cases depending upon $p$.\\

\noindent\textbf{Case:} $\mathbf{p=2.}$ For $(n,r)=(216,29)$, we compute that the slope of the right-most edge of $NP_2(\Lnr)$ is $M_2=1$. Thus for $(n,r)=(216,29)$, $I_2=\{1\}$.

For $(n,r)=(4320,55)$, we compute that $M_2=\mu_{32}=\frac{17}{16}<2$ and that the right-most edge has vertices $(n-32, \nu_2(b_n-32))$ and $(n,\nu_2(b_n))$. Thus the second-last edge of $NP_2(\Lnr)$ (which lies before the right-most edge) has slope 
\begin{align*}
\max\limits_{33\leq  j\leq n} \left\{ \frac{\nu_2(b_{n-32})-\nu_2(b_{n-j})}{j-32} \right\}.
\end{align*}
For each $33\leq j \leq n$, we calculate that
\begin{align*}
\frac{\nu_2(b_{n-2})-\nu_2(b_{n-j})}{j-32} &=  \frac{\nu_2((n+r-32)!) - \nu_2((n+r-j)!) +\nu_2(\binom{n}{32}) - \nu_2(\binom{n}{j})}{j-32}\\
& <1.
\end{align*}
Therefore for $(n,r)=(4320,55)$, $I_2=\phi$.\\

\noindent\textbf{Case:} $\mathbf{p=3.}$ For $(n,r)=(216,29)$, we compute that $M_3=1$. Thus for $(n,r)=(216,29)$, $I_3=\{1\}$.

For $(n,r)=(4320,55)$, we compute that the slope of the right-most edge of $NP_3(\Lnr)$ is $M_3=\mu_2=2$ and that the right-most edge has vertices $(n-2, \nu_3(b_n-2))$ and $(n,\nu_3(b_n))$. Thus the second-last edge of $NP_3(\Lnr)$ (which lies before the right-most edge) has slope 
	\begin{align*}
	\max\limits_{3 \leq j \leq n} \left\{ \frac{\nu_3(b_{n-2})-\nu_3(b_{n-j})}{j-2} \right\}.
	\end{align*}
For each $3\leq j \leq n$, we calculate that
	\begin{align*}
	\frac{\nu_3(b_{n-2})-\nu_3(b_{n-j})}{j-2} &=  \frac{\nu_3((n+r-2)!) - \nu_3((n+r-j)!) +\nu_3(\binom{n}{2}) - \nu_3(\binom{n}{j})}{j-2}\\
	& <1.
	\end{align*}
Therefore for $(n,r)=(4320,55)$, $I_3=\{2\}$.\\

\noindent\textbf{Case:} $\mathbf{p=5.}$ For $(n,r)=(216,29)$, we compute that $M_5=1$. Thus for $(n,r)=(216,29)$, $I_5=\{1\}$. 

For $(n,r)=(4320,55)$, we compute that $M_5=\mu_1=3$. That is, the right-most edge has vertices $(n-1, \nu_5(b_n-1))$ and $(n,\nu_5(b_n))$ and thus the second-last edge of $NP_5(\Lnr)$ (which lies before the right-most edge) has slope 
	\begin{align*}
	\max\limits_{2 \leq j \leq n} \left\{ \frac{\nu_5(b_{n-1})-\nu_5(b_{n-j})}{j-1} \right\}.
	\end{align*}
For each $2\leq j \leq n$, we calculate that
	\begin{align*}
	\frac{\nu_5(b_{n-1})-\nu_5(b_{n-j})}{j-1} &=  \frac{\nu_5((n+r-1)!) - \nu_5((n+r-j)!) +\nu_5(n) - \nu_5(\binom{n}{j})}{j-1}\\
	& <1.
	\end{align*}
Therefore for $(n,r)=(4320,55)$, $I_5=\{3\}$.\\

\noindent\textbf{Case:} $\mathbf{p=7.}$ For $(n,r)=(216,29)$, we compute that $M_7=\mu_1=2$. Thus for $(n,r)=(216,29)$, $I_7=\{0,2\}$. So the right-most edge has vertices $(n-1, \nu_7(b_n-1))$ and $(n,\nu_7(b_n))$ and thus the second-last edge of $NP_7(\Lnr)$ (which lies before the right-most edge) has slope 
	\begin{align*}
	\max\limits_{2 \leq j \leq n} \left\{ \frac{\nu_7(b_{n-1})-\nu_7(b_{n-j})}{j-1} \right\}.
	\end{align*}
For each $2\leq j \leq n$, we calculate that
	\begin{align*}
	\frac{\nu_7(b_{n-1})-\nu_7(b_{n-j})}{j-1} &=  \frac{\nu_7((n+r-1)!) - \nu_7((n+r-j)!) +\nu_7(n) - \nu_7(\binom{n}{j})}{j-1}\\
	& <1.
	\end{align*}
Therefore for $(n,r)=(216,29)$, $I_7=\{2\}$.

For $(n,r)=(4320,55)$, we compute that $M_7=1$. Thus for $(n,r)=(4320,55)$, $I_7=\{1\}$.\\

\noindent\textbf{Case: } $\mathbf{p>7.}$ For $(n,r)\in \{(216,29),(4320,55)\}$, by looking at the prime factorization of $ \binom{n+r}{r} $, we find that 
	\begin{align*}
	\nu_p \left( \binom{n+r}{r} \right) \leq 2.
	\end{align*}
For $j\geq p$, by Remark \ref{remark}, we have 
	\begin{align*}
	\mu_j &\leq \frac{j-\s_p(j)}{(p-1)j} + \frac{1}{j} \nu_p \left( \binom{n+r}{r} \right)\\
	&\leq \frac{1}{(p-1)} + \frac{2}{j}\\
	&\leq \frac{1}{(p-1)} + \frac{2}{p} < \frac{3}{p-1} \leq \frac{3}{10}< 1.
	\end{align*}
For $2<j<p$, by Remark \ref{remark} again, we have 
	\begin{align*}		\mu_j &\leq \frac{j-\s_p(j)}{(p-1)j} + \frac{1}{j} \nu_p \left( \binom{n+r}{r} \right)\\
	&=\frac{1}{j} \nu_p \left( \binom{n+r}{r} \right) \leq \frac{2}{j} < 1.
	\end{align*}
For each prime $p>7$ and $1\leq j\leq2$, we verify by exact computation that $\mu_j<1$. Therefore $M_p<1$, i.e., $I_p \subseteq\{0\}$ for each $p>7$. Hence for each $p>7$, $p$ cannot divide $a$, i.e., $p\nmid a$.\\

Let $(n,r)=(216,29)$. Then $\nu_2(a)=\nu_3(a)=\nu_5(a)=1$ and  either $7\nmid a$ or $\nu_7(a)=2$. Hence $a\in \{30, 1470\}$. We verify that $x=-30$ and $x=-1470$ do not satisfy $\mathcal{L}_{n}^{<r>}(x)=0$.

Let $(n,r)=(4320,55)$. Then $2\nmid a$,  $\nu_3(a)=2$, $\nu_5(a)=3$ and $\nu_7(a)=1$. Hence $a=7875$. We verify that $x=-7875$ does not satisfy $\mathcal{L}_{n}^{<r>}(x)=0$. 
\qed

\section{Galois groups of $\lnr$: Proof of Theorem \ref{mainthm1}$(ii)$} \label{proofpart2}
In this section, we prove Theorem \ref{mainthm1}$(ii)$ by describing the Galois groups of $\lnr$ for $23\leq r\leq 60$ and $n \geq 1$. From Section \ref{proofpart1}, we have
$\lnr$ is irreducible for each $23\leq r \leq 60$ and $n\geq1$.

For $23\leq r\leq 60$, let $B_r$ be given by
\begin{align*}
	B_{23}&=B_{24}=\dots =B_{28}=\{1,2,\dots,31\},\\
	B_{29}&=B_{30}=\{1,2,\dots,33\},\\
	B_{31}&=B_{32}=\dots =B_{36}=\{1,2,\dots,39\},\\
	B_{37}&=B_{38}=\dots= B_{40}=\{1,2,\dots,43\},\\
	B_{41}&=B_{42}=\{1,2,\dots,45\},\\
	B_{43}&=B_{44}=\dots =B_{46}=\{1,2,\dots,49\},\\
	B_{47}&=B_{48}=\dots =B_{52}=\{1,2,\dots,55\},\\
	B_{53}&=B_{54}=\dots =B_{58}=\{1,2,\dots,61\},\\
	B_{59}&=B_{60}=\{1,2,\dots,63\}.
\end{align*}

For each $23\leq r\leq60$ and $n\in B_r$, we compute $G_n(r)$ using MAGMA online, and in fact, $G_n(r)=A_n$ for $(n,r)\in \{(4,24),(5,28),(24,25),(25,24),(28,23),(28,29),(32,33),$\\ $(33,36),(36,37),(40,41),(44,45),(48,49),(48,51),(49,48),(49,50),(52,53),(56,57)\}$\\ and $G_n(r)=S_n$ otherwise.

From now onwards, we assume that $n\notin B_r$. We first show that $G_n(r)$ contains $A_n$.

Fix $r$ with $23\leq r\leq 60$. We have $\max\{48-r,8+\frac{5r}{3}\}=8+\frac{5r}{3}$. Let
$$C_r=\{n \in \N : n< 8+\frac{5r}{3} \textrm{ and } \nexists \textrm{ a prime } p \textrm{ with } \frac{n+r}{2}<p<n-2 \}.$$
Observe that $C_r$ is finite and $B_r \subseteq C_r$.
By Lemma \ref{lemma for gp} $(i)$ and $(ii)$,  we have $G_n(r)$ contains $A_n$ for each $n\notin C_r$. For $n \in C_r$, we now apply Lemma \ref{our result on gp} to get $G_n(r)$ contains $A_n$ for each $n\in C_r$, $n\notin B_r$. Hence $G_n(r)$ contains $A_n$ for $n\notin B_r$.

Thus, by Lemma \ref{lemma for gp}$(iii)$, we have
	\begin{align*}
	G_n(r)=	\begin{cases}
		A_n & \textrm{if } \Dt \textrm{is a square,}\\
		S_n & \textrm{otherwise.}\\
		\end{cases}
	\end{align*}
Therefore to complete the proof of Theorem \ref{mainthm1}$(ii)$, it suffices to check if $\Dt$ is a square or not. In fact, we show that for each $23\leq r\leq60$ and $n\notin B_r$, $\Dt$ is never a square.

For integers $a$ and $b$, we write $a\sim b$ if $a=bc^2$ for some integer $c>0$. Let $\square$ denote the square of an unspecified non-zero integer. We consider the following cases:\\

\noindent\textbf{Case 1.} $n$ is odd:
We have
$$\Dt \sim (-1)^{n(n-1)/2}(1\cdot3\cdot5\cdots n)(n+r-1)(n+r-3)\cdots(r+2).$$
If $n \equiv 3 \pmod 4$, then $\Dt$ is not a square. Thus assume $n \equiv 1 \pmod 4$.

\noindent\textbf{Subcase 1(a).} r is even:
By re-arranging the factors, we see that
$$\Dt \sim (1\cdot3\cdot5\cdots (r-1))((r+1)(r+2)\cdots n)(n+1)(n+3)\cdots(n+r-1).$$

For $n > \frac{3(r-1)}{2}$, we have
	\begin{align*}
	\frac{n+r-1}{2}<\frac{5}{6}n.
	\end{align*}
By Lemma \ref{harborthkemnitz} with $x=\frac{5}{6}n$, there is a prime $p$ satisfying
	\begin{align*}
	\frac{n+r-1}{2}<p<n
	\end{align*}
so that $\nu_p(\Dt)$ is odd, and hence $\Dt$ is not a square.

For $n \leq \frac{3(r-1)}{2}$ with $n\notin B_r$, we check directly that $\Dt$ is not a square.

\noindent\textbf{Subcase 1(b).} r is odd:
By re-arranging the factors, we see that
\begin{align}
\Dt \sim (1\cdot3\cdot5\cdots r)(n+2)(n+4)\cdots(n+r-1). \label{d1}
\end{align}
If $n\leq1089$, then $n+r-1 \leq 1148$ and since there are at least 10 consecutive odd integers in $\{n+2,n+4,\dots,n+r-1\}$, it follows from Lemma \ref{primefact} that there is a prime $p$ in this set. We note that $n\notin B_r$ implies $n \geq r+4$ and thus we have
	\begin{align*}
	r\leq n+2\leq p \leq n+r-1<2n<2p.
	\end{align*}
Hence we get $\nu_p(\Dt)$ is odd. Therefore $\Dt$ is not a square.

Now suppose that $n>1089$ and $\Dt$ is a square. For fixed odd $23\leq r\leq 60$, we focus on the expression on the right hand side of \eqref{d1} and find the squarefree integer $y_r$ such that 
\begin{align*}
1\cdot3\cdot5\cdots r= y_r \times \square.
\end{align*}
Thus for $x_{n,r}=y_r(n+2)(n+4)\cdots(n+r-1)$, we have
\begin{align*}
(1\cdot3\cdot5\cdots r)(n+2)(n+4)\cdots(n+r-1) =x_{n,r} \times \square
\end{align*}
so that $\Dt \sim x_{n,r}$, i.e., $\Dt$ is a square if and only if $x_{n,r}$ is a square. We give the list of $x_{n,r}$ for odd $r$ in the range $23\leq r\leq60$  in Table \ref{table1}.

\begin{table}[h!]
\caption{List of $r$ and $x_{n,r}$ where $\Dt \sim x_{n,r}$ }
\label{table1}
\begin{tabular}{c l}
	\hline
	$r$ & \hspace{4 cm}$x_{n,r}$ \\
	\hline
	23&$(3\cdot11\cdot13\cdot17\cdot19\cdot23)(n+2)(n+4)\cdots(n+22)$\\
	25& $(3\cdot11\cdot13\cdot17\cdot19\cdot23)(n+2)(n+4)\cdots(n+24)$\\
	27&$(11\cdot13\cdot17\cdot19\cdot23)(n+2)(n+4)\cdots(n+26)$\\
	29&$(11\cdot13\cdot17\cdot19\cdot23\cdot29)(n+2)(n+4)\cdots(n+28)$\\
	31&$(11\cdot13\cdot17\cdot19\cdot23\cdot29\cdot31)(n+2)(n+4)\cdots(n+30)$\\
	33&$(3\cdot13\cdot17\cdot19\cdot23\cdot29\cdot31)(n+2)(n+4)\cdots(n+32)$\\
	35&$(3\cdot5\cdot7\cdot13\cdot17\cdot19\cdot23\cdot29\cdot31)(n+2)(n+4)\cdots(n+34)$\\
	37&$(3\cdot5\cdot7\cdot13\cdot17\cdot19\cdot23\cdot29\cdot31\cdot37)(n+2)(n+4)\cdots(n+36)$\\
	39&$(5\cdot7\cdot17\cdot19\cdot23\cdot29\cdot31\cdot37)(n+2)(n+4)\cdots(n+38)$\\
	41&$(5\cdot7\cdot17\cdot19\cdot23\cdot29\cdot31\cdot37\cdot41)(n+2)(n+4)\cdots(n+40)$\\
	43&$(5\cdot7\cdot17\cdot19\cdot23\cdot29\cdot31\cdot37\cdot41\cdot43)(n+2)(n+4)\cdots(n+42)$\\
	45&$(7\cdot17\cdot19\cdot23\cdot29\cdot31\cdot37\cdot41\cdot43)(n+2)(n+4)\cdots(n+44)$\\
	47&$(7\cdot17\cdot19\cdot23\cdot29\cdot31\cdot37\cdot41\cdot43\cdot47)(n+2)(n+4)\cdots(n+46)$\\
	49&$(7\cdot17\cdot19\cdot23\cdot29\cdot31\cdot37\cdot41\cdot43\cdot47)(n+2)(n+4)\cdots(n+48)$\\
	51&$(3\cdot7\cdot19\cdot23\cdot29\cdot31\cdot37\cdot41\cdot43\cdot47)(n+2)(n+4)\cdots(n+50)$\\
	53&$(3\cdot7\cdot19\cdot23\cdot29\cdot31\cdot37\cdot41\cdot43\cdot47\cdot53)(n+2)(n+4)\cdots(n+52)$\\
	55&$(3\cdot5\cdot7\cdot11\cdot19\cdot23\cdot29\cdot31\cdot37\cdot41\cdot43\cdot47\cdot53)(n+2)(n+4)\cdots(n+54)$\\
	57&$(5\cdot7\cdot11\cdot23\cdot29\cdot31\cdot37\cdot41\cdot43\cdot47\cdot53)(n+2)(n+4)\cdots(n+56)$\\
	59&$(5\cdot7\cdot11\cdot23\cdot29\cdot31\cdot37\cdot41\cdot43\cdot47\cdot53\cdot59)(n+2)(n+4)\cdots(n+58)$\\
	\hline
\end{tabular}
\end{table}

Let $r$ be a fixed odd integer in the range $23\leq r \leq 60$. For any prime $p$, there are at most $\left\lfloor \frac{r-3}{2p}\right\rfloor +1$ terms in the set $\{n+2,n+4,\dots,n+r-1\}$ divisible by $p$. For each prime $7\leq p \leq r$ appearing in $\frac{x_{n,r}}{\prod\limits_{1\leq i\leq \frac{r-1}{2}}(n+2i)}$, we delete those terms in $\{n+2,n+4,\dots,n+r-1\}$ divisible by $p$. We find that there are at least $6$ terms in this set of the form $ax^2$ with $a\in\{1,3,5,15\}$ and it follows that there are two distinct terms in $\{n+2,n+4,\dots,n+r-1\}$ whose product is a square. This contradicts Lemma \ref{squares} for $m=n$ and $t=r-1$. Thus $x_{n,r}$ and hence $\Dt$ is not a square. We give the following three examples to illustrate this argument.

Let $r=23$. Then
$$\Dt \sim x_{n,r}= (3\cdot11\cdot13\cdot17\cdot19\cdot23)(n+2)(n+4)\cdots(n+22).$$
There are at most $5$ terms in $\{n+2,n+4,\dots,n+22\}$ which are divisible by $11,13,17,19$ or $23$. After removing these terms, we are left with at least $6$ terms each of the form $ax^2$ with $a\in\{1,3\}$. Therefore there are two distinct terms in $\{n+2,n+4,\dots,n+22\}$ whose product is a square. This contradicts Lemma \ref{squares} for $m=n$ and $t=r-1$. Therefore $\Dt$ is not a square.

Let $r=27$. Then
$$\Dt \sim x_{n,r}= (11\cdot13\cdot17\cdot19\cdot23)(n+2)(n+4)\cdots(n+26).$$
There are at most $4$ terms in $\{n+2,n+4,\dots,n+26\}$ which are divisible by $13,17,19$ or $23$ and further $11$ divides at most $2$ terms of this set. After removing these terms, we are left with at least $7$ terms in this set which are squares. This contradicts Lemma \ref{squares} for $m=n$ and $t=r-1$. Thus $x_{n,r}$ and hence $\Dt$ is not a square.

Let $r=37$. Then
$$\Dt \sim x_{n,r}= (3\cdot5\cdot7\cdot13\cdot17\cdot19\cdot23\cdot29\cdot31\cdot37)(n+2)(n+4)\cdots(n+36).$$
The number of terms in $\{n+2,n+4,\dots,n+36\}$ divisible by $7,13$ and $17$ are at most $3,2$ and $2$ respectively. Also each of $19,23,29,31$ and $37$ divides at most one term in this set. After removing these terms, we are left with at least $6$ terms in the set $\{n+2,n+4,\dots,n+36\}$ each of which is of the form $ax^2$ with $a\in \{1,3,5,15\}$ and it follows that there are two distinct terms in $\{n+2,n+4,\dots,n+36\}$ whose product is a square. We get a contradiction using Lemma \ref{squares} as above.\\

\noindent\textbf{Case 2.} $n$ is even:
We have
$$\Dt \sim (-1)^{n(n-1)/2}(1\cdot 3\cdot 5\cdots (n-1))(n+r-1)(n+r-3)\cdots(r+1).$$
If $n \equiv 2 \pmod 4$, then $\Dt$ is not a square.
Thus assume $n \equiv 0 \pmod 4$.

\noindent\textbf{Subcase 2(a).} r is odd:
By re-arranging the factors, we see that
$$\Dt \sim (1\cdot3\cdot5\cdots (r-2))(r(r+1)\cdots n)(n+2)(n+4)\cdots(n+r-1).$$

For $n > \frac{3(r-1)}{2}$, we have
	\begin{align*}
	\frac{n+r-1}{2}<\frac{5}{6}n.
	\end{align*}
By Lemma \ref{harborthkemnitz} with $x=\frac{5}{6}n$, there is a prime $p$ satisfying
	\begin{align*}
	\frac{n+r-1}{2}<\frac{5}{6}n<p<n
	\end{align*}
so that $\nu_p(\Dt)$ is odd, and hence $\Dt$ is not a square.

For $n \leq \frac{3(r-1)}{2}$ with $n\notin B_r$, we check directly that $\Dt$ is not a square.

\noindent\textbf{Subcase 2(b).} r is even:
By re-arranging the factors, we see that
\begin{align}
\Dt \sim (1\cdot3\cdot5\cdots (r-1))(n+1)(n+3)\cdots(n+r-1).  \label{d2}
\end{align}
If $n\leq1089$, then $n+r-1 \leq 1148$ and since there are at least 10 consecutive odd integers in $\{n+1,n+3,\dots,n+r-1\}$, it follows from Lemma \ref{primefact} that there is a prime $p$ in this set. We note that $n\notin B_r$ implies $n \geq r+4$ and thus we have
\begin{align*}
r\leq n+2\leq p \leq n+r-1<2n<2p.
\end{align*}
Thus $\nu_p(\Dt)$ is odd and hence $\Dt$ is not a square.

For fixed even $23\leq r\leq 60$, we focus on the expression on the right hand side of \eqref{d2} and find the squarefree integer $y_r$ such that 
\begin{align*}
1\cdot3\cdot5\cdots (r-1)= y_r \times \square.
\end{align*}
Thus for $x_{n,r}=y_r(n+1)(n+3)\cdots(n+r-1)$, we have
\begin{align*}
(1\cdot3\cdot5\cdots (r-1))(n+1)(n+3)\cdots(n+r-1) =x_{n,r} \times \square
\end{align*}
so that $\Dt \sim x_{n,r}$. We give the list of $x_{n,r}$ for even $r$ in the range $23\leq r\leq60$ in Table \ref{table2}.

\begin{table}[h]
\caption{List of $r$ and $x_{n,r}$ where $\Dt \sim x_{n,r}$  }
\label{table2}
\begin{tabular}{c l}
	\hline
	$r$ & \hspace{4 cm}$x_{n,r}$ \\
	\hline
	24 & $(3\cdot11\cdot13\cdot17\cdot19\cdot23)(n+1)(n+3)\cdots(n+23)$\\
	26& $(3\cdot11\cdot13\cdot17\cdot19\cdot23)(n+1)(n+3)\cdots(n+25)$\\
	28&$(11\cdot13\cdot17\cdot19\cdot23)(n+1)(n+3)\cdots(n+27)$\\
	30&$(11\cdot13\cdot17\cdot19\cdot23\cdot29)(n+1)(n+3)\cdots(n+29)$\\
	32&$(11\cdot13\cdot17\cdot19\cdot23\cdot29\cdot31)(n+1)(n+3)\cdots(n+31)$\\
	34&$(3\cdot13\cdot17\cdot19\cdot23\cdot29\cdot31)(n+1)(n+3)\cdots(n+33)$\\
	36&$(3\cdot5\cdot7\cdot13\cdot17\cdot19\cdot23\cdot29\cdot31)(n+1)(n+3)\cdots(n+35)$\\
	38&$(3\cdot5\cdot7\cdot13\cdot17\cdot19\cdot23\cdot29\cdot31\cdot37)(n+1)(n+3)\cdots(n+37)$\\
	40&$(5\cdot7\cdot17\cdot19\cdot23\cdot29\cdot31\cdot37)(n+1)(n+3)\cdots(n+39)$\\
	42&$(5\cdot7\cdot17\cdot19\cdot23\cdot29\cdot31\cdot37\cdot41)(n+1)(n+3)\cdots(n+41)$\\
	44&$(5\cdot7\cdot17\cdot19\cdot23\cdot29\cdot31\cdot37\cdot41\cdot43)(n+1)(n+3)\cdots(n+43)$\\
	46&$(7\cdot17\cdot19\cdot23\cdot29\cdot31\cdot37\cdot41\cdot43)(n+1)(n+3)\cdots(n+45)$\\
	48&$(7\cdot17\cdot19\cdot23\cdot29\cdot31\cdot37\cdot41\cdot43\cdot47)(n+1)(n+3)\cdots(n+47)$\\
	50&$(7\cdot17\cdot19\cdot23\cdot29\cdot31\cdot37\cdot41\cdot43\cdot47)(n+1)(n+3)\cdots(n+49)$\\
	52&$(3\cdot7\cdot19\cdot23\cdot29\cdot31\cdot37\cdot41\cdot43\cdot47)(n+1)(n+3)\cdots(n+51)$\\
	54&$(3\cdot7\cdot19\cdot23\cdot29\cdot31\cdot37\cdot41\cdot43\cdot47\cdot53)(n+1)(n+3)\cdots(n+53)$\\
	56&$(3\cdot5\cdot7\cdot11\cdot19\cdot23\cdot29\cdot31\cdot37\cdot41\cdot43\cdot47\cdot53)(n+1)(n+3)\cdots(n+55)$\\
	58&$(5\cdot7\cdot11\cdot23\cdot29\cdot31\cdot37\cdot41\cdot43\cdot47\cdot53)(n+1)(n+3)\cdots(n+57)$\\
	60&$(5\cdot7\cdot11\cdot23\cdot29\cdot31\cdot37\cdot41\cdot43\cdot47\cdot53\cdot59)(n+1)(n+3)\cdots(n+59)$\\
	\hline
\end{tabular}
\end{table}
Let $r$ be a fixed even integer in the range $23\leq r \leq 60$.  For any prime $p$, there are at most $\left\lfloor \frac{r-2}{2p}\right\rfloor +1$ terms in the set $\{n+1,n+3,\dots,n+r-1\}$ divisible by $p$. For each prime $7\leq p \leq r$ appearing in $\frac{x_{n,r}}{\prod\limits_{1\leq i\leq r/2}(n+2i-1)}$, we delete those terms in $\{n+1,n+3,\dots,n+r-1\}$ divisible by $p$. We find that there are at least $6$ terms in this set of the form $ax^2$ with $a\in\{1,3,5,15\}$ and it follows that there are two distinct terms in $\{n+1,n+3,\dots,n+r-1\}$ whose product is a square. This contradicts Lemma \ref{squares} for $m=n-1$ and $t=r$. Thus $x_{n,r}$ and hence $\Dt$ is not a square. We illustrate this argument for $r=36$.

Let $r=36$. Then
$$\Dt \sim x_{n,r}= (3\cdot5\cdot7\cdot13\cdot17\cdot19\cdot23\cdot29\cdot31)(n+1)(n+3)\cdots(n+35).$$
There are at most $4$ terms in $\{n+1,n+3,\dots,n+35\}$ which are divisible by $19,23,29$ or $31$ and further each of $13$ and $17$ divides at most $2$ terms of this set and $7$ divides at most $3$ terms of this set. After removing these terms, we are left with at least $7$ terms of the form $ax^2$ with $a\in \{1,3,5,15\}$ and it follows that there are two distinct terms in $\{n+1,n+3,\dots,n+35\}$ whose product is a square. This contradicts Lemma \ref{squares} for $m=n-1$ and $t=r$. Therefore $\Dt$ is not a square.

\qed

\section{Proof of Theorem \ref{mainthm2}} \label{proof of mainthm2}
Suppose that $\lnr$ has a factor of degree $k$. Then by Lemma \ref{rk1}, $k<\frac{r}{1.63}$. By Lemma \ref{degree}, we have $n_0 \leq k <\frac{r}{1.63}$. Thus if $p|n_0$, then $p^{\nu_p(n_0)}<r$ and in fact $p^{\nu_p(n)}=p^{\nu_p(n_0)}<r$. Also by Lemma \ref{ppower}, if $p|n_1$, then $p^{\nu_p(n)}\leq r$. Hence
	\begin{align*}
	n =n_0n_1 =\prod\limits_{p|n}p^{\nu_p(n)} \leq \prod\limits_{p\leq r} r = r^{\pi(r)} = e^{\pi(r) {\rm ~log~} r} \leq e^{r\left(1+\frac{1.2762}{\log  r}\right)}
	\end{align*}
by Lemma \ref{pi}.

It remains to show that if $n>e^{r\left(1+\frac{1.2762}{\log  r}\right)}$, then $G_n(r)$ contains $A_n$. By Lemma \ref{lemma for gp}$(ii)$, this is the case if
\begin{align}
e^{r\left(1+\frac{1.2762}{\log  r}\right)}> \max\{48-r,8+\frac{5r}{3}\} \label{eq3}
\end{align}
By Theorem \ref{mainthm1} and the results of Schur, Hajir, Sell, Nair and Shorey stated in the introduction, we may assume that $r \geq61$.  Then $\max\{48-r,8+\frac{5r}{3}\}=8+\frac{5r}{3}<2r$. From $e^x > \frac{x^2}{2}$ for $x>0$, we have $e^{r\left(1+\frac{1.2762}{\log  r}\right)}>\frac{r^2}{2}>2r$ and hence the assertion \eqref{eq3} follows. This proves Theorem \ref{mainthm2}.
\qed

\section*{Acknowlegdements}
The second author like to thank the funding agencies DST, India and DRDO, India for supporting this work under DST Fast Track Project and DRDO CARS Project. The authors are also thankful to the anonymous referee for carefully analyzing the article and for giving kind and thoughtful remarks on an earlier version of the paper.


\begin{thebibliography}{100}
    \bibitem{bffl} P. Banerjee, M. Filaseta, C. E. Finch and J. R. Leidy, \emph{On classifying Laguerre polynomials which have Galois group the alternating group},
      J. Th\'{e}or. Nombres Bordeaux, {\bf 25} (2013), 1--30.

	\bibitem{coleman1987} R. F. Coleman, {\it On the Galois groups of the exponential Taylor polynomials}, L'Enseignement Math. 33 (1987), 183-189.
	
	\bibitem{dumas} G. Dumas, {\it Sur quelques cas d'irr\'{e}ductibilit\'{e} des polyn\^{o}mes \'{a} coefficients rationnels}, Journal de Math. Pure et Appl. 2 (1906), 191-258.
	
	\bibitem{dusart} P. Dusart, {\it In\'egalit\'es explicites pour $\psi(x),~ \theta(x),~ \pi(x)$ et les nombres premiers}, C. R. Math. Acad. Sci. Soc. R.
	Canada 21 (1999), 53-59.

	\bibitem{filaseta1995} M. Filaseta, {\it The irreducibility of all but finitely many Bessel polynomials}, Acta Math. 174 (1995), 383-397.
	
	\bibitem{filaseta_trifonov2002} M. Filaseta and O. Trifonov, {\it The Irreducibility of the Bessel polynomials}, J. Reine Angew. Math. 550 (2002), 125-140.
	
    \bibitem{fkt} M. Filaseta, T. Kidd and O. Trifonov, \emph{Laguerre polynomials with Galois group $A_m$ for each $m$}, J. Number Theory, {\bf 132} (2012), no. 4, 776--805.

	\bibitem{grosswald} E. Grosswald, {\it Bessel Polynomials}, Lecture Notes in Math. 698, Springer, Berlin, 1978.
	
	\bibitem{hajir1995} F. Hajir, {\it Some $\tilde{A_n}$-extensions obtained from generalized Laguerre polynomials}, J. Number Theory 50 (1995), 206-212.
		
	\bibitem{hajir2005} F. Hajir, {\it On the Galois group of generalized Laguerre Polynomials}, J. Th\'{e}or. Nombres Bordeaux, 17(2005), no. 2, 517-525.
	
	\bibitem{hajir2009} F. Hajir, {\it Algebraic properties of a family of generalized Laguerre polynomials}, Canad. J. Math., 61 (2009), 583-603.
	
	\bibitem{harborthkemnitz} H. Harborth and A. Kemnitz, {\it Calculations for Bertrands postulate}, Math. Mag., 54 (1981), 33-34.
	
	\bibitem{legendre} A. M. Legendre, {\it Essai sur la Th\'eorie des Nombres}, Paris, 1808.
	
	\bibitem{nair} S. G. Nair and T.N Shorey, {\it Irreducibility of Laguerre Polynomial $L_{n}^{(-1-n-r)}(x)$ }, Indagationes Mathematicae, 26(2015), 615-625.
	
	\bibitem{sell2004} E. A. Sell, {\it On a certain family of generalized Laguerre polynomials}, J. Number Theory 107 (2004), 266-281.
	
	\bibitem{schur1973a} I. Schur, {\it Gleichungen Ohne Affekt. In: Gesammelte Abhandlungen}, Band III, Springer-Verlag, Berlin-New York, (1973), 191-197.
		
	\bibitem{schur1973b} I. Schur, {\it Affektlose Gleichungen in der Theorie der Laguerreschen und Hermiteschen Polynome In: Gesammelte Abhandlungen}, Band III, Springer-Verlag, Berlin-New York, (1973), 227-233.
\end{thebibliography}
\end{document}